\documentclass[a4paper,10pt,english]{smfart}
\usepackage{amsfonts}
\usepackage{amsmath} 
\usepackage{hyperref} 
\usepackage{latexsym}
\usepackage{array}
\usepackage{amssymb}
\usepackage{smfthm}
\usepackage{bull}
\theoremstyle{plain}

\newtheorem*{acknowledgements}{Acknowledgements}

\newcommand{\R}{  \mathbb{R}   }

\newcommand{\eps}{\varepsilon}
\newcommand{\e}{  \text{e}   }

\newcommand{\Z}{  \mathbb{Z}   }
\newcommand{\N}{  \mathbb{N}   }

\renewcommand{\H}{  \mathcal{H}   }

\newcommand{\T}{  \mathbb{T}   }
\newcommand{\dd}{  \text{d}   }

\newcommand{\ov}{  \overline  }
\renewcommand{\a}{  \alpha   }
\renewcommand{\b}{  \beta   }
\newcommand{\s}{  \sigma   }
\newcommand{\wh}{  \widehat   }
\newcommand{\<}{  \langle   }
\renewcommand{\>}{  \rangle   }
\numberwithin{equation}{section}
\author{ Laurent Thomann}
\address{Universit\'e de Nantes, Laboratoire de Math\'ematiques J. Leray, UMR CNRS 6629\\
2, rue de la Houssini\`ere \\
F-44322 Nantes Cedex 03, France. }
\email{laurent.thomann@univ-nantes.fr}
\urladdr{http://www.math.sciences.univ-nantes.fr/$\sim$thomann/}

\title[Quadratic  NLS on $\T$]{Low regularity for  a quadratic Schr\"odinger equation on $\T$} 

\begin{document}
\frontmatter
 \begin{abstract}
In this paper we consider a Schr\"odinger equation on the circle with  a quadratic nonlinearity. Thanks to an explicit computation of the first Picard iterate, we give a precision on the dynamic of the solution, whose existence was proved by C. E. Kenig, G. Ponce and L. Vega \cite{KPV}. We also show that  the equation is well-posed in a space $\H^{s,p}(\T)$ which contains the Sobolev space $H^{s}(\T)$ when $p\geq 2$.  
\end{abstract}

 \begin{altabstract}
Dans cet article on s'int\'eresse \`a une \'equation de Schr\"odinger sur le cercle avec une non-lin\'earit\'e quadratique. Un calcul explicite de la premi\`ere it\'er\'ee de Picard permet de donner une pr\'ecision sur la dynamique de la solution, dont l'existence a \'et\'e d\'emontr\'ee par C. E. Kenig, G. Ponce et L. Vega \cite{KPV}. On montre \'egalement que l'\'equation est bien pos\'ee dans un espace $\H^{s,p}(\T)$ qui contient l'espace de Sobolev $H^{s}(\T)$ lorsque $p\geq 2$.
\end{altabstract}

\subjclass{35A07; 35B35 ; 35B45; 35Q55}
\keywords{Non linear Schr\"odinger equation, rough initial conditions}
\altkeywords{\'equation de  Schr\"odinger non lin\'eaire, donn\'ees irr\'eguli\`eres}
\thanks{The author was supported in part by the  grant ANR-07-BLAN-0250.}

\maketitle
\mainmatter


\section{Introduction}

Denote by $\T=\R/2\pi\Z$  the unidimensional torus. In this paper we consider the following nonlinear Schr\"odinger equation
\begin{equation}\label{NLS}
\left\{
\begin{aligned}
&i\partial_t u+\Delta u   = \kappa \overline{u}^{2},\quad \kappa=\pm 1,\;\;
(t,x)\in\R\times \mathbb{T},\\
&u(0,x)= f(x)\in X,
\end{aligned}
\right.
\end{equation}
where $X$ is a Banach space (the space of the initial conditions). \\[3pt]
This equation has been intensively studied in the case $x\in M$ where $M$ is a Riemannian manifold and for different nonlinearities, usually of the form 
\begin{equation*}
F(u,\ov{u})=\pm u^{p_{1}}\,\ov{u}^{p_{2}}, \quad \text{where} \quad p_{1},p_{2}\in \N. 
\end{equation*}
Here we mainly discuss the results in one dimension for quadratic nonlinearities. For the other cases see \cite{CW}, \cite{KPV}, \cite{BGT2}, and references therein.\\[5pt]
\subsection{Previous results on the real line}~\\[10pt]
 In the case $x\in \R$, J. Ginibre and G. Velo \cite{GV}, Y. Tsutsumi \cite{Tsutsumi}, T. Cazenave and F.\:B. Weissler \cite{CW}  showed that the Cauchy problem is well posed for $f\in L^{2}(\R)$, for every nonlinearity of the type \eqref{NLin} with $p_{1}+q_{1}\leq 5$. The proof relies on the use of Strichartz inequalities, which are of the form 
\begin{equation}\label{NLin}
\|\e^{it\Delta}f\|_{L^{p}(\R,L^{q}(\R))}\leq C\|f\|_{L^{2}(\R)},\quad \text{with}\quad \frac1p+\frac2q=\frac12.
\end{equation}
In \cite{KPV}, C. E. Kenig, G. Ponce, and L. Vega show that \eqref{NLS} is well posed in $X=H^{s}(\R)$ :
\begin{itemize}
\item for $s>-3/4$ in the case $F(u,\ov{u})=\pm u^{2}$ or $F(u,\ov{u})=\pm \ov{u}^{2}$ ;
\item for $s>-1/4$ in the case $F(u,\ov{u})=\pm |u|^{2}$.
\end{itemize}
To obtain these results, the authors prove some bilinear estimates in the conormal spaces $X^{s,b}$ (see Definition \ref{defi}), and they also show that these estimates are optimal, and as a consequence it is impossible to perform a usual fixed point argument in these spaces, below the threshold $s=-3/4$ (resp. $s=-1/4$). Notice that the $X^{s,b}$ spaces distinguish the structure of the nonlinearity, which was not the case for the Strichartz spaces.\\[3pt]
In \cite{BejenaruTao}, I. Bejenaru and T. Tao extend the well posedness results to $s\leq -1$ in the case $F(u,\ov{u})= u^{2}$, and show that the equation \eqref{NLS} is ill-posed in $H^{s}(\R)$ when $s<-1$.\\[3pt]
\subsection{Previous results on the torus}~\\[5pt]
In the case $x\in \T$, J. Bourgain \cite{Bourgain1:res} established the embedding $X^{0,3/8}\subset L^4_{x,t}$, which permitted to show that the problem \eqref{NLS} is locally well posed in $L^{2}(\T)$, for every nonlinearity \eqref{NLin} with $p_{1}+p_{2}\leq 3$.\\[3pt]
Then, C. E. Kenig, G. Ponce, and L. Vega   \cite{KPV}, thanks to bilinear estimates in $X^{s,b}$ (see Theorem \ref{theoKPV} below),  obtained the well posedness of  \eqref{NLS} in $H^{s}(\T)$ for $s>-1/2$ in the case $F(u,\ov{u})=\pm u^{2}$ or $F(u,\ov{u})=\pm \ov{u}^{2}$. Again, these estimates fail if $s<-1/2$. \\[5pt]
%
%
\subsection{The $\H^{s,p}(\T)$ and $X^{s,b}$ spaces}~\\[5pt]
Now we introduce the $\H^{s,p}(\T)$ spaces
 \begin{defi}\label{defi2}($\H^{s,p}$ spaces)\\
For $s\in \R$ and $p\geq 1$, denote by $\H^{s,p}=\H^{s,p}(\T)$ the completion of $\mathcal{C}^{{\infty}}(\T)$ with respect to the norm
\begin{equation*}
\|f\|_{\H^{s,p}}=\Big(\,\sum_{n\in \Z}\<n\>^{ps}\,|{\breve{f}}(n)|^{p}\,\Big )^{\frac1p}.
\end{equation*}
Here $\breve{f}(n)$ denotes the Fourier coefficient of $f$ (see \eqref{FourierCoef}).
\end{defi}
\noindent These spaces where introduced by L. H\"ormander (see \cite{Hormander}, Section 10.1).\\[5pt]
\noindent There are several motivations to introduce these spaces\\[4pt]
\noindent $\bullet$ First notice that $\H^{s,2}(\T)=H^{s}(\T)$, and for $p> 2$ we have the (strict) inclusion $H^{s}(\T)\subset \H^{s,p}(\T)$.\\[2pt]
 $\bullet$ Then, the space $ \H^{s,p}$ scales like $H^{s(p)}$ where $s(p)=-\frac12+s+\frac1p$. Hence, if $s(p)<-\frac12$, the space $\H^{s,p}$ contains elements $f$ such that $|\breve{f}(n)|\longrightarrow +\infty$ when $n\longrightarrow +\infty$. Therefore we can go closer to the scaling of the equation \eqref{NLS} which is $-\frac32$.\\[2pt]
 $\bullet$ T. Cazenave, L. Vega and M. C. Vilela \cite{CaVeVi} where the first authors to study nonlinear Schr\"odinger equations in $\H^{s,p}$-like spaces. In fact they show that a class of NLS equations on $\R^{N}$ is well-posed if the linear flow belongs to some weak $L^{p}$ space. Moreover they prove that this condition can be ensured if the initial data $f$ satisfies $\widehat{f}\in L^{p,\infty}(\R^{N})$ for some $p\geq 1$. This latter space is a continuous version of the space $\H^{s,p}$. \\[2pt]
 $\bullet$ In  \cite{Grun2} A. Gr\"unrock  establishes bilinear and trilinear estimates in conormal spaces $X^{s,b}_{p,q}$ (see definition below) based on $L^{r}$. This permits him to show that the cubic Schr\"odinger equation
 \begin{equation*}
\left.
\begin{aligned}
&i\partial_t u+\Delta u   =\pm |u|^{2} u,
\,\, (t,x)\in\R\times \mathbb{R},
\end{aligned}
\right.
\end{equation*}
is well-posed for initial conditions in the corresponding continuous version of the space $\H^{s,p}$. He obtains  analogous results for the DNLS equation \cite{Grun2} and for the mKdV equation \cite{Grun1}.\\[5pt] 
In \cite{Christ}, M. Christ shows that the modified cubic problem
\begin{equation*}
\left\{
\begin{aligned}
&i\partial_t u+\Delta u   =\pm\big(|u|^{2}-2\mu(|u|^{2})   \big)u,
\,\, \text{where}\,\, 
\mu(|v|^{2})=\frac1{2\pi}\int_{-\pi}^{\pi}|v(x)|^{2}\text{d}x,\\
&u(0,x)= f(x)\in \H^{s,p}(\T),
\end{aligned}
\right.
\end{equation*}
is well posed in $\H^{s,p}(\T)$ for any $s\geq 0$ and $p\geq 1$. See \cite{Christ} for precise statements.\\[5pt]
Recently, A. Gr\"unrock and S. Herr \cite{GrunHerr} have shown  the well-posedness in $\H^{s,p}$ spaces of the DNLS equation on the torus, thanks to multilinear estimates.\\[5pt]
See \cite{Christ, Grun1,Grun2, GrunHerr} for other features of the spaces $\H^{s,p}$ and more references.\\[10pt]
 $\bullet$ Notice that the $\H^{s,p}$ is preserved by the linear Schr\"odinger flow. Write
 \begin{equation*}
 f(x)=\sum_{n\in \Z}\a_{n}\e^{inx},\;\;\text{then}\;\;\;  \e^{it\Delta}f(x)=\sum_{n\in \Z}\a_{n}\e^{-in^{2}t}\e^{inx},
 \end{equation*}
 and for all $t\in \R$, $\|e^{it\Delta}f\|_{\H^{s,p}}=\|f\|_{\H^{s,p}}$.\\[10pt]
\noindent We now define the $X^{s,b}$ spaces
\begin{defi}\label{defi}($X^{s,b}$ spaces)\\
(i) For $s,b\in \R$, denote by $X^{s,b}=X^{s,b}(\R\times \T)$ the completion of $\mathcal{C}^{{\infty}}\big(\T,\mathcal{S}(\R)\big)$ with respect to the norm
\begin{equation*}
\|F\|_{X^{s,b}}=\big(\sum_{n \in \Z} \int_{\R}\<\tau+n^{2}\>^{2b}\<n\>^{2s}|\widetilde{F}(\tau,n)|^{2}\text{d}\tau\big )^{\frac12}.
\end{equation*}
(ii) Let $T>0$, we define the restriction spaces $X_{T}^{s,b}=X^{s,b}([-T,T]\times \T)$ by 
\begin{equation}
\|F\|_{X_{T}^{s,b}}=\inf\Big\{\|\psi(\frac{t}{T})\, F\|_{X^{s,b}}, \;F\in X^{s,b}\;\text{with}\;\psi\in \mathcal{S}(\R)\;\text{s.t.}\;\psi|_{[-1,1]}=1\Big\}.
\end{equation}
Here $\widetilde{F}$ stands for the space-time Fourier transform (see \eqref{SpTi}).
\end{defi}
\noindent In the following, we will mainly use the space $X_{1}^{s,b}=X^{s,b}([-1,1]\times \T)$.\\[5pt]
\noindent We recall the key estimates which permit to perform a fixed point argument in the $X^{s,b}$ spaces, and to deduce that the equation \eqref{NLS} is well posed in $H^{s}$ for $s>-\frac12$.

\begin{prop}\label{EstInt} Let $s\leq 0$ and $\frac12<b\leq 1$. 
 Then  for all   $F\in {X^{s,b-1}_{1}}$, we have
\begin{equation*}
\|\int_{0}^{t}\e^{i(t-t')\Delta}F(t',\cdot)\text{d}t'\|_{X^{s,b}_{1}}  \leq C \| F\|_{X^{s,b-1}_{1}}\;.
\end{equation*}
\end{prop}
\noindent See \cite{Ginibre} for a proof. Notice this estimate holds in the general case of a riemannian manifold, indeed the proof reduces to time integrations. Notice also that we always have the estimate 
\begin{equation*}
\|\e^{it\Delta}f\|_{X^{s,b}_{1}}  \leq C \| f\|_{H^{s}},
\end{equation*}
but we won't use it in this paper. \\[8pt]
The following theorem is one of the main results of \cite{KPV} (see Theorem 1.9. in \cite{KPV})

\begin{theo}\label{theoKPV}( Kenig-Ponce-Vega \cite{KPV})
Let $-\frac12<s\leq 0$, then there exists $b_{0}>\frac12$ such that for all $\frac12<b\leq b_{0}$ and all $v,w\in X^{s,b}(\R\times \T)$
\begin{equation}\label{FondaKPV}
\|\ov{v}\,\ov{w}\|_{X^{s,b-1}}\lesssim \|v\|_{X^{s,b}}\|w\|_{X^{s,b}}.
\end{equation}
Moreover, for any $s<-\frac12$ and $b\in \R$, an estimate of the form \eqref{FondaKPV} fails.
\end{theo}
\noindent We can deduce the following

\begin{coro}
Let $-\frac12<s\leq 0$, then there exists $b_{0}>\frac12$ such that for all $\frac12<b\leq b_{0}$ and all $v,w\in X^{s,b}([-1,1]\times \T)$
\begin{equation}\label{FondaKPV2}
\|\ov{v}\,\ov{w}\|_{X_{1}^{s,b-1}}\lesssim \|v\|_{X^{s,b}_{1}}\|w\|_{X^{s,b}_{1}}.
\end{equation}
\end{coro}

\begin{proof}
Let $\psi_{1},\psi_{2} \in \mathcal{C}^{\infty}_{0}(\R)$ be so that $\psi_{1},\psi_{2}=1$ on $[-1,1]$ and $\text{supp}\;\psi_{1},\psi_{2} \subset [-2,2]$. Then by  \eqref{FondaKPV} applied to $\psi_{1}(t)v$ and  $\psi_{2}(t)w$, we obtain
\begin{equation*}
\|\ov{v}\,\ov{w}\|_{X_{1}^{s,b-1}}\leq \|\psi_{1}({t })\ov{v}\,\psi_{2}({t })\ov{w}\|_{X^{s,b-1}}\lesssim \|\psi_{1}v\|_{X^{s,b}}\|\psi_{2}w\|_{X^{s,b}},
\end{equation*}
and the result follows, by choosing $\psi_{1}$ and $\psi_{2}$ which realise the infimum for the $X^{s,b}([-1,1]\times \T)$ norm.
\end{proof}

\section{Main results of this paper}~\\
\subsection{Local well posedness in the Sobolev scale}~\\[10pt]
Our first result is a precision on the dynamic of the solution of \eqref{NLS} when the initial condition $f$ is in $ H^{s_{0}}(\T)$ with $-\frac12<s_{0}\leq 0$.

\noindent Let  $f\in\mathcal{D}'(\T)$. Then define
\begin{equation*}
u_{0}(t,x)=\e^{it\Delta}f(x)=\sum_{n\in \Z}\breve{f}(n)\e^{-in^{2}t}\e^{inx},
\end{equation*}
the free Schr\"odinger evolution and
\begin{equation*}
u_{1}(t,x)=-i\int_{0}^{t}\e^{i(t-t')\Delta}(\overline{u_{0}}^{2})(t',x)\text{d}t',
\end{equation*}
 the first Picard iterate of the equation \eqref{NLS}. Then we will show that there exists $b>\frac12$ so that 
\begin{equation}\label{EstPic0}
\|u_{1}\|_{ X^{0,b}([-1,1]\times\T)}\lesssim  \|f\|^{2}_{H^{s_{0}}(\T)}.
\end{equation}
Hence, $u_{1}$ is more regular than $f$ : there  is a  gain of $|s_{0}|$ derivative. We will take profit of this phenomenon to prove that it is also the case for $u-\e^{it\Delta}f$, where $u$ is the solution  of \eqref{NLS}. 

\begin{theo}\label{theo1}
Let $\kappa=\pm 1$. Let $-\frac12<s_{0}\leq 0$ and $f\in H^{s_{0}}(\T)$. Then  there exist $b>\frac12$ and $T>0$
 such that there exists a unique solution $u$ to \eqref{NLS} in the space
\begin{equation}\label{espace0}
Y_{T}^{0,b}=\Big(\e^{it\Delta}f+X^{0,b}\big([-T,T]\times \T\big)\Big).
\end{equation}
Moreover, given $0<T'<T$ there exist $R=R(T')>0$ such that the map $\tilde{f} \mapsto \tilde{u}(t)$ from $\{\;\tilde{f}\in H^{s_{0}}(\T)\;:\; \|\tilde{f}-f\|_{H^{s_{0}}}<R\;\}$ into the class \eqref{espace0} with $T'$ instead of $T$ is Lipschitz.
\end{theo}

\noindent This result will be obtained with a contraction argument in the space $X^{0,b}$ (thanks to the gain of regularity), and therefore we will only need the estimate \eqref{FondaKPV} with $s=0$.\\[5pt]
%
\subsection{Local well posedness in the $\H^{s,p}$ scale}~\\[10pt]
We can use the gain of regularity of the first Picard iterate to solve the Cauchy problem \eqref{NLS} for data $f\in\H^{s,p}(\T)$, and this will improve slightly the result of \cite{KPV}, as we have the inclusion $H^{s_{0}}(\T)\subset \H^{s_{0},p}(\T)$ for $p>0$.\\[5pt]
The following condition on the real numbers $s_{0}$ and $p$ will be needed for our result
\begin{equation}\label{cond}
\frac3p+s_{0}>\frac56.
\end{equation}

\begin{theo}\label{theo2}
Let $\kappa=\pm 1$. Let $s_{0}>-\frac12$ and let  $p>2$ be so that the condition \eqref{cond} is satisfied. Let $f\in \H^{s_{0},p}(\T)$. Then for all $s_{1}<-1+\frac2p$ there exist $b>\frac12$, $s_{1}<s<-1+\frac2p$, and $T>0$ such that there exists a unique solution $u$ to \eqref{NLS} in the space
\begin{equation}\label{espace}
Y_{T}^{s,b}=\Big(\e^{it\Delta}f+X^{s,b}\big([-T,T]\times \T\big)\Big).
\end{equation}
Moreover, given $0<T'<T$ there exist $R=R(T')>0$ such that the map $\tilde{f} \mapsto \tilde{u}(t)$ from $\{\;\tilde{f}\in \H^{s_{0},p}(\T)\;:\; \|\tilde{f}-f\|_{\H^{s_{0},p}}<R\;\}$ into the class \eqref{espace} with $T'$ instead of $T$ is Lipschitz.
\end{theo}
\noindent To prove Theorem \ref{theo2} we will use the estimate \eqref{FondaKPV} in its full strength.\\[5pt]
\noindent From the previous result, we can immediately deduce

\begin{coro}\label{corotheo}
Let $\a<\frac1{18}$ and let $f\in \mathcal{D}'(\T)$ be such that $|\breve{f}(n)|\lesssim \<n\>^{\a}$.  Then there exist $s>-\frac19$,  $b>\frac12$ and $T>0$ such that  there  exists a unique solution to \eqref{NLS} in the space
\begin{equation*}
Y_{T}^{s,b}=\Big(\e^{it\Delta}f+X^{s,b}\big([-T,T]\times \T\big)\Big).
\end{equation*}
\end{coro}

\noindent For instance :  Let $0<\eps<1$ be small and  $\a=\frac1{18}-\eps$. Define  $f\in\mathcal{D}'(\T)$ by $\breve{f}(n)= \<n\>^{\a}$. Then  $f\in H^{s}(\T)$ for $s<-\frac12-\frac1{18}+\eps<-\frac12$, but $f\in \H^{s_{0},p}(\T)$ for some $(s_{0},p)$ which satisfies the assumptions of Theorem \ref{theo2}

\begin{rema}
The result of Theorem \ref{theo2} is interesting when $s_{0}$ is close to $-\frac12$, and $p$ as big as possible, under the assumption \eqref{cond}.  \\
 Let $0<\eps<1$ be small and set $s_{0}=-\frac12+\eps$. Then $p>2$ satisfies \eqref{cond}  iff 
\begin{equation*}
\frac49-\frac13\eps<\frac1p<\frac12.
\end{equation*}
Hence, the parameter $s$ in Theorem \ref{theo2} can be chosen close to $-\frac19$. In other words there is a gain of $\sim \frac12-\frac19=\frac7{18}$ derivative.
\end{rema}

\begin{rema}
The conclusions of Theorem \ref{theo1} and \ref{theo2} are likely to hold with the nonlinearities $F(u)=\pm u^{2}$, but we do not pursue this here.
\end{rema}
~\\[5pt]
\subsection{Notations and plan of the paper}~\\[10pt]
\noindent For $F\in \mathcal{S}(\R)$ we define the time-Fourier transform by
\begin{equation*}
\widehat{F}(\tau)=\int_{\R}\e^{-i\tau t}F(t)\text{d}t,
\end{equation*}
which has the following properties
\begin{equation}\label{TF}
 \wh{\ov{F}}(\tau)=\ov{\wh{F}}(-\tau)\quad \text{and} \quad\wh{F\e^{i\theta \cdot}}(\tau)=\wh{F}(\tau-\theta)\;\;\text{for all}\;\;\theta\in \R.
\end{equation}
Each  $F\in \mathcal{C}^{{\infty}}\big(\T,\mathcal{S}(\R)\big)$ admits the Fourier  expansion 
\begin{equation}\label{FourierCoef}
F(t,x)=\sum_{n\in \Z}\breve{F}(t,n)\e^{inx}, \;\; \text{where}\;\; \breve{F}(\tau,n)=\frac1{2\pi}\int_{-\pi}^{\pi}\e^{-in x}F(t,x)\text{d}x,
\end{equation}
is the periodic Fourier coefficient of $F$.\\
Finally, we denote by
\begin{equation}\label{SpTi}
\widetilde{F}(\tau,n)=\frac1{2\pi}\int_{\R}\int_{-\pi}^{\pi}\e^{-i(\tau t+n x)}F(t,x)\text{d}t\text{d}x,
\end{equation} 
 the space-time Fourier transform.\\[5pt]

\begin{enonce*}{Notations}
In this paper $c$, $C$ denote constants the value of which may change
from line to line. These constants will always be universal, or depending only on fixed quantities. We use the notations $a\sim b$,  
$a\lesssim b$ if $\frac1C b\leq a\leq Cb$, $a\leq Cb$
respectively.
\end{enonce*}

\noindent In Section \ref{FirstPicard} we make explicit computations to estimate the first Picard iteration in $X^{s,b}$ spaces.\\
Then, in Section \ref{Bilin} we establish a bilinear estimate in   $X^{s,b}$ spaces.\\
In Section \ref{contrac}, we follow an idea of N. Burq and N. Tzvetkov \cite{BT2, BT3} and look for a solution of \eqref{NLS} of the form $u=\e^{it\Delta}f+v$. The existence and uniqueness of $v$ is then proved with  a fixed point argument, using the estimates of the previous sections.
\begin{acknowledgements}
The author would like to thank N. Burq and N. Tzvetkov for useful discussions on the subject.
\end{acknowledgements}


\section{The first Picard iteration}\label{FirstPicard}~\\[10pt]
\begin{lemm}\label{integrale}
Let $\varphi\in \mathcal{S}(\R)$. Then
\begin{equation*}
\int_{\R}\frac{1}{\<\tau+A\>}|\varphi|(\tau)\dd\tau \lesssim \frac1{\<A\>},
\end{equation*}
uniformly in $A\in \R$.
\end{lemm}

\begin{proof}
As ${\varphi}$ is in the Schwartz class $|{\varphi}|(\tau)\lesssim \<\tau\>^{-3}$. \\
Then notice that $\<\tau\>\<\tau+A\>\gtrsim \<A\>$,
therefore 
\begin{equation*}
\int_{\R}\frac{\<A\>}{\<\tau+A\>}|{\varphi}|(\tau)\dd\tau\lesssim\int_{\R}\frac{\<A\>}{\<\tau\>\<\tau+A\>} \frac1{\<\tau\>^{2}}\dd\tau\lesssim1,
\end{equation*}
hence the result.
\end{proof}

\noindent Let  $f\in\mathcal{D}'(\T)$, denote by $\a_{n}=\breve{f}(n)$. Then define
\begin{equation}\label{u_{0}}
u_{0}(t,x)=\e^{it\Delta}f(x)=\sum_{n\in \Z}\a_{n}\e^{-in^{2}t}\e^{inx},
\end{equation}
the free Schr\"odinger evolution and
\begin{equation}\label{u_{1}}
u_{1}(t,x)=-i\int_{0}^{t}\e^{i(t-t')\Delta}(\overline{u_{0}}^{2})(t',x)\text{d}t',
\end{equation}
which is the first Picard iterate of the equation \eqref{NLS}.

\begin{prop}\label{propPicard}
Let $-\frac12<s_{0}\leq 0$ and $p\geq 2$.  Then there exists $b_{1}> \frac12$ such that for all $\frac12<b<b_{1}$, all $f\in \H^{s_{0},p}(\T)$ and all $s<-1+2/p$ we have 
\begin{equation}\label{EstPic}
\|u_{1}\|_{ X^{s,b}([-1,1]\times\T)}\lesssim  \|f\|^{2}_{\H^{s_{0},p}(\T)}.
\end{equation}
Moreover, in the case $p=2$, the estimate \eqref{EstPic} holds for $s=0$.
\end{prop}
\begin{rema}
The result of Proposition \ref{propPicard} shows that the first Picard iterate is more regular than the initial condition, when 
$s_{0}$ is close to $-\frac12$ and $p<4$. In this case, we can take $s>s_{0}$.\\
The result we stated is not optimal when $s_{0}$ is far from $-\frac12$.
\end{rema}

\begin{proof}
Let $b>\frac12$ to be chosen later. Denote by $\b=2(1-b)<1$ and $\s=-s\geq 0$. \\
Let $\psi_{0}\in \mathcal{C}^{{\infty}}_{0}(\R)$ s.t. $\psi_{0}=1$ on $[-1,1]$, and $\psi\in \mathcal{C}^{{\infty}}_{0}(\R)$ s.t. $\psi_{0}\psi=\psi_{0}$. Then by Definition \ref{defi} and Proposition \ref{EstInt} we have 
\begin{eqnarray}
\|u_{1}\|_{ X^{s,b}([-1,1]\times\T)}&\leq &\|\psi_{0}(t)\,u_{1}\|_{ X^{s,b}(\R\times\T)}\nonumber\\
    &\lesssim &\|\psi(t)\,\overline{u_{0}}^{2}\|_{ X^{s,b-1}(\R\times\T)}.\label{somme}
\end{eqnarray}
Now by the expression \eqref{u_{0}}, we have (with the change of variables $p=-n-m$)
\begin{eqnarray*}
\psi(t)(\overline{u_{0}}^{2})&=&\psi(t)\sum_{(n,m)\in \Z^{2}}\overline{\alpha_{n}}\,\ov{\alpha_{m}}\,\,\e^{i(n^{2}+m^{2})t}\e^{-i(n+m)x}\\
&=&\psi(t)\,\sum_{p\in \Z}\Big(\sum_{n\in \Z}\overline{\alpha_{n}}\,\ov{\alpha_{-n-p}}\,\,\e^{i(n^{2}+(n+p)^{2})t}\Big)\e^{ipx}.
\end{eqnarray*}
Hence we we deduce the Fourier coefficients of $\psi(t)(\overline{u_{0}}^{2})$ :
\begin{equation}\label{c_{p}}
c_{p}(t):=\sum_{n\in \Z}\overline{\alpha_{n}}\,\ov{\alpha_{-n-p}}\,\,\e^{i(n^{2}+(n+p)^{2})t}
=\psi(t)\breve{(\overline{u_{0}}^{2})}(p).
\end{equation}
From the properties \eqref{TF} of the time-Fourier transform, we deduce
\begin{equation}\label{defcp}
\widehat{c_{p}}(\tau)=\sum_{n\in \Z}\ov{\alpha_{n}}\,\ov{\alpha_{-n-p}}\;\widehat{\psi}(\tau-n^{2}-(n+p)^{2}),
\end{equation}
and by Definition \ref{defi}, we have
\begin{equation*}
I:=\|\psi(t)(\overline{u_{0}}^{2})\|^{2}_{{X^{s,b-1}(\R\times \T)}}=\sum_{p\in \Z}\int_{\R}\<\tau+p^{2}\>^{-\beta}\<p\>^{2s}|\widehat{c_{p}}(\tau)|^{2}\dd\tau,
\end{equation*}
with $\beta =2(1-b)$. Now, by  Lemma \ref{cp} (see below for the statement and proof) we have  
\begin{equation*}
|\wh{c_{p}}(\tau)|^{2}\lesssim \sum_{n\in \Z} |\a_{n}|^{2}|\a_{-n-p}|^{2}|\widehat{\psi}|(\tau -n^{2}-(n+p)^{2}),
\end{equation*}
uniformly in $(\tau,p)\in \R\times \Z$. With the change of variables $m=-n-p$ and $\tau'=\tau-n^{2}-m^{2}$, we deduce
\begin{eqnarray}
I&\lesssim&\sum_{n\in\Z}\sum_{p\in\Z}\int_{\R}\frac{\<p\>^{2s}}{\<\tau+p^{2}\>^{\beta}}|\alpha_{n}|^{2}|\alpha_{-n-p}|^{2}|\widehat{\psi}|(\tau-n^{2}-(n+p)^{2})\dd\tau\nonumber\\
&=&\sum_{n\in\Z}\sum_{m\in\Z}\int_{\R}\frac{\<n+m\>^{2s}}{\<\tau+(n+m)^{2}\>^{\beta}}|\alpha_{n}|^{2}|\alpha_{m}|^{2}|\widehat{\psi}|(\tau-n^{2}-m^{2})\dd\tau\nonumber\\
&=&\sum_{(n,m)\in\Z^{2}}\int_{\R}\frac{\<n+m\>^{2s}}{\<\tau+(n+m)^{2}+n^{2}+m^{2}\>^{\beta}}|\alpha_{n}|^{2}|\alpha_{m}|^{2}|\widehat{\psi}|(\tau)\dd\tau.\label{I1}
\end{eqnarray}
Apply Lemma \ref{integrale} with $A=(n+m)^{2}+n^{2}+m^{2}$. Denote by $\s=-s\geq0$. Then from \eqref{I1} we deduce

\begin{equation}\label{i}
I\lesssim\sum_{(n,m)\in\Z^{2}}\frac{|\alpha_{n}|^{2}|\alpha_{m}|^{2}}{\<n+m\>^{2\s}\<n^{2}+m^{2}\>^{\beta}}.
\end{equation}
$\bullet$ From here we assume that $\s>0$.\\[3pt]
For $m\in\Z$, denote by 
\begin{equation*}
\gamma_{m}=\sum_{n\in\Z}\frac{|\alpha_{n}|^{2}}{\<n+m\>^{2\s}\<n\>^{\b}},
\end{equation*}
thanks to the inequality $\<n^{2}+m^{2}\>\geq \<n\>\<m\>$, from \eqref{i} we deduce
\begin{equation}\label{ii}
I\lesssim\sum_{m\in\Z}\Bigg(\frac{|\alpha_{m}|^{2}}{\<m\>^{\b}}\Big(\sum_{n\in\Z}\frac{|\alpha_{n}|^{2}}{\<n+m\>^{2\s}\<n\>^{\b}}\Big)\Bigg)= \sum_{m\in\Z}\gamma_{m}\, \frac{|\alpha_{m}|^{2}}{\<m\>^{\b}}.
\end{equation}
Now by H\"older, for $p\geq 2$
\begin{equation}\label{iii}
 \sum_{m\in\Z}\gamma_{m}\, \frac{|\alpha_{m}|^{2}}{\<m\>^{\b}}\lesssim \Big(\sum_{k\in\Z}\frac{|\alpha_{k}|^{p}}{\<k\>^{\b p/2}}\;\Big)^{\frac2p}\Big(\sum_{m\in\Z}\gamma_{m}^{q_{1}}\Big)^{\frac1{q_{1}}}=\|f\|^{2}_{\H^{-\b/2,p}}\Big(\sum_{m\in\Z}\gamma_{m}^{q_{1}}\Big)^{\frac1{q_{1}}},
\end{equation}
with
\begin{equation}\label{j}
\frac1{q_{1}}=1-\frac2{p}.
\end{equation}
To estimate the last term in \eqref{iii}, we observe that 
\begin{equation*}
\gamma_{m}=\Big(\,\frac{|\a_{k}|^{2}}{\<k\>^{\b}}*\frac{1}{\<j\>^{2\s}}\Big)(m),
\end{equation*}
then by Young's inequality, for all $p_{1},r_{1}\geq 1$ so that 
\begin{equation}\label{jj}
\frac1{q_{1}}=\frac1{p_{1}}+\frac1{r_{1}}-1,
\end{equation}
 and so that  for $2\s r_{1}>1$, we have
\begin{equation}\label{iv}
\Big(\sum_{m\in\Z}\gamma_{m}^{q_{1}}\Big)^{\frac1{q_{1}}}\lesssim \Big(\sum_{k\in\Z}\frac{|\alpha_{k}|^{2p_{1}}}{\<k\>^{\b p_{1}}}\;\Big)^{\frac1{p_{1}}}\Big(\sum_{j\in\Z}\frac{1}{\<j\>^{2\s r_{1}}}\Big)^{\frac1{r_{1}}}.
\end{equation}
We take $p_{1}=p/2$. This choice together with the conditions \eqref{j}, \eqref{jj} and  $2\s r_{1}>1$ yields
\begin{equation*}
\s>\frac1{2r_{1}}=1-\frac2p,
\end{equation*}
and thus by \eqref{ii},  \eqref{iii} and  \eqref{iv} we obtain
\begin{equation*}
I\lesssim \|f\|^{4}_{\H^{-\b/2,p}}.
\end{equation*}
Now we choose $b>\frac12$ such that $\b=-2s_{0}$, i.e. $b=2(1-\b)=1+s_{0}$, and thus $\frac12<b\leq 1$, as we assumed that $-\frac12<s_{0}\leq 0$.\\
Together with \eqref{somme}, this concludes the proof of the first statement of Proposition \ref{propPicard}.\\[5pt]
$\bullet$ Now we deal with the case  $p=2$ and $\s=0$.\\[3pt]
  By \eqref{i} we only have to bound the term
\begin{equation*}
J:=\sum_{(n,m)\in\Z^{2}}\frac{|\alpha_{n}|^{2}|\alpha_{m}|^{2}}{\<n^{2}+m^{2}\>^{\beta}}.
\end{equation*}
Thanks to the inequality $\<n^{2}+m^{2}\>\geq \<n\>\<m\>$, we get
\begin{equation*}
J\leq \sum_{(n,m)\in\Z^{2}}\frac{|\alpha_{n}|^{2}|\alpha_{m}|^{2}}{\<n\>^{\beta}\<m\>^{\beta}}=\|f\|^{4}_{H^{s_{0}}},
\end{equation*}
which was the claim.
\end{proof}

\begin{lemm}\label{cp}
Let $\wh{c_{p}}(\tau)$ be defined by \eqref{defcp}. Then there exists $C>0$, which only depends on $\psi$, so that
\begin{equation}\label{lemcp}
|\wh{c_{p}}(\tau)|^{2}\leq C \sum_{n\in \Z} |\a_{n}|^{2}|\a_{-n-p}|^{2}|\widehat{\psi}|(\tau -n^{2}-(n+p)^{2}),
\end{equation}
for   all $(\tau,p)\in \R\times \Z$.
\end{lemm}

\begin{proof}
Denote by 
\begin{equation*}
\wh{\psi_{1}}(\tau, n,p)=\widehat{\psi}(\tau -n^{2}-(n+p)^{2}),
\end{equation*}
then

\begin{equation*}
|\wh{c_{p}}(\tau)|^{2}=\sum_{(n,m)\in \Z^{2}} \ov{\a_{n}}\;\ov{\a_{-n-p}} \; \a_{m}\;\a_{-m-p}\;\wh{\psi_{1}}(\tau, n,p) \;\wh{\psi_{2}}(\tau,m,p),
\end{equation*}
and with the change of variables $m=n+k$ we obtain
\begin{equation}\label{square}
|\wh{c_{p}}(\tau)|^{2}=\sum_{n\in \Z} \sum_{k\in \Z} \ov{\a_{n}}\;\ov{\a_{-n-p}} \; \a_{n+k}\;\a_{-n-k-p}\;\wh{\psi_{1}}(\tau, n,p) \;\wh{\psi_{2}}(\tau,n+k,p).
\end{equation}
As $\wh{\psi}\in \mathcal{S}(\R)$, for all $N\in \N$, $|\wh{\psi}|\lesssim \<\tau\>^{-N}$. In the remaining of the proof, the constant $N$ may change from line to line. By the inequality $\<A+B\>\lesssim \<A\>\<B\>$, we have 

\begin{multline} 
 |\wh{\psi_{1}}(\tau, n,p)\wh{\psi_{2}}(\tau,n+k,p)|\lesssim   \\[4pt]
 \begin{aligned}
&\lesssim\frac{|\wh{\psi_{1}}(\tau, n,p)|^{\frac12} \;|\wh{\psi_{2}}(\tau,n+k,p)|^{\frac12}}{\big\<\tau-n^{2}-(n+p)^{2}\big\>^{N}  \<\tau-(n+k)^{2}-(n+k+p)^{2}\>^{N}  }     \\[4pt]
&\lesssim \frac{|\wh{\psi_{1}}(\tau, n,p)|^{\frac12} \;|\wh{\psi_{1}}(\tau,n+k,p)|^{\frac12}}{ \<2k\,(2n+k+p)\>^{N} }. \label{prod}   
\end{aligned}
\end{multline}
\noindent $\bullet$ If $k=0$ or $k=-2n-p$, in the sum \eqref{square}, we immediately get the bound \eqref{lemcp}.\\[3pt]
$\bullet$
Denote by
\begin{equation*}
I_{p}(\tau)=\sum_{n\in \Z} \sum_{k\in \Z_{p}^{*}} \ov{\a_{n}}\;\ov{\a_{-n-p}} \; \a_{n+k}\;\a_{-n-k-p}\;\wh{\psi_{1}}(\tau, n,p) \;\wh{\psi_{2}}(\tau,n+k,p).
\end{equation*}
where $\Z_{p}^{*}=\Z\backslash \{0,-2n-p\}$.\\
 If $k\neq0$ and  $k\neq-2n-p$, observe that
\begin{equation*}
\big \<2k\,(2n+k+p)\big\>^{2}\gtrsim \<k\> \<2n+k+p\>,
 \end{equation*}
thus by \eqref{prod}

\begin{equation*}
|\wh{\psi_{1}}(\tau, n,p) \;\wh{\psi_{1}}(\tau,n+k,p)| \lesssim  \frac{|\wh{\psi_{1}}(\tau, n,p)|^{\frac12} \;|\wh{\psi_{1}}(\tau,n+k,p)|^{\frac12}}{ \<k\>^{N} \<2n+k+p\>^{N}     },
\end{equation*}
and 
\begin{eqnarray*}
I_{p}(\tau)&\lesssim &\sum_{n\in \Z}|\a_{n}||\a_{-n-p}||\wh{\psi_{1}}(\tau, n,p)|^{\frac12}\Bigg(\sum_{k\in \Z}     \frac{  |\a_{n+k}||\a_{-n-k-p}| |\wh{\psi_{1}}(\tau,n+k,p)|^{\frac12}}{ \<k\>^{N} \<2n+k+p\>^{N}     }
  \Bigg)\\[2pt]
  &= &\sum_{n\in \Z}|\a_{n}||\a_{-n-p}||\wh{\psi_{1}}(\tau, n,p)|^{\frac12}\Big(\sum_{j\in \Z}     \frac{|\a_{j}||\a_{-j-p}||\wh{\psi_{1}}(\tau,j,p)|^{\frac12}}{ \<n-j\>^{N} \<n+j+p\>^{N}     } \Big),
\end{eqnarray*}
after the change of variables $j=k+n$ in the second sum. \\
Now by Cauchy-Schwarz
\begin{eqnarray*}
\sum_{j\in \Z}     \frac{|\a_{j}||\a_{-j-p}||\wh{\psi_{1}}(\tau,j,p)|^{\frac12}}{ \<n-j\>^{N} \<n+j+p\>^{N}     }
&\lesssim& d(\tau,p)^{\frac12}\Big(\sum_{l\in \Z} \frac{1}{\<n-l\>^{N}}  \frac{1}{\<n+l+p\>^{N}}   \Big)^{\frac12},
\end{eqnarray*}
where 
\begin{equation*}
d(\tau,p)= \sum_{j\in \Z}     |\a_{j}|^{2}|\a_{-j-p}|^{2}|\wh{\psi_{1}}(\tau,j,p)|,
\end{equation*}
and as $\<n-l\>\<n+l+p\>\gtrsim \<2n+p\>$,
\begin{eqnarray*}
\sum_{l\in \Z} \frac{1}{\<n-l\>^{N}}  \frac{1}{\<n+l+p\>^{N}} &\lesssim& 
 \frac{1}{\<2n+p\>^{N}} \sum_{l\in \Z} \frac{1}{\<n-l\>^{N}}  \frac{1}{\<n+l+p\>^{N}}\\
 &\lesssim&  \frac{1}{\<2n+p\>^{N}}, 
\end{eqnarray*}
by Cauchy-Schwarz. Thus 
\begin{eqnarray*}
I_{p}(\tau)&\lesssim &d(\tau,p)^{\frac12} \sum_{n\in \Z}|\a_{n}||\a_{-n-p}||\wh{\psi_{1}}(\tau, n,p)|^{\frac12}   \frac{1}{\<2n+p\>^{N}}    \\
&\lesssim&d(\tau,p)^{\frac12}\Big( \sum_{n\in \Z}|\a_{n}|^{2}|\a_{-n-p}|^{2}|\wh{\psi_{1}}(\tau, n,p)|\Big)^{\frac12}
\Big( \sum_{n\in \Z}  \frac{1}{\<2n+p\>^{N}}\Big)^{\frac12}\\
&\lesssim&d(\tau,p),
\end{eqnarray*}
which completes the proof.
\end{proof}

\section{The bilinear estimate}\label{Bilin}

This section is devoted to the proof of the following result

\begin{prop}\label{bilinear}
Let $-\frac12<s_{0}\leq 0$ and $p\geq 2$. Then for all
\begin{equation}\label{condprop}
-\frac16-s_{0}-\frac1p<s\leq 0,
\end{equation}
there exists $b_{2}>\frac12$ such that for all $\frac12<b<b_{2}$, all $f\in \H^{s_{0},p}(\T)$ and all $v\in X^{s,b}_{1}(\R\times \T)$ 
\begin{equation}\label{Estproduit}
\|\int_{0}^{t}\e^{i(t-t')\Delta}\ov{u_{0}}\,\ov{v}(t',\cdot)\text{d}t'\|_{X^{s,b}([-1,1]\times \T)}  \lesssim 
\|f\|_{ \H^{s_{0},p}}\|v\|_{X^{s,b}([-1,1]\times \T)}\,,
\end{equation}
where $u_{0}(t)=\e^{it\Delta}f$.
\end{prop}

\noindent Proposition \ref{bilinear} shows that, under condition \eqref{condprop},  the term
\begin{equation*}
\int_{0}^{t}\e^{i(t-t')\Delta}\ov{u_{0}}\,\ov{v}(t',\cdot)\text{d}t',
\end{equation*}
has the regularity of $v$, even if $f$ is less regular. For instance, with $p=2$ and $s=0$, we obtain 
\begin{equation*}
\|\int_{0}^{t}\e^{i(t-t')\Delta}\ov{u_{0}}\,\ov{v}(t',\cdot)\text{d}t'\|_{X^{0,b}_{1}}  \lesssim 
\|f\|_{ H^{s_{0}}}\|v\|_{X^{0,b}_{1}},
\end{equation*}
whenever $s_{0}>-\frac12-\frac16$.\\[5pt]
\noindent We now state a few technical results.\\[2pt]
\noindent We will need the following lemma which is proved in \cite{KPV}.

\begin{lemm}\label{LemmaKPV}
If $\gamma>\frac12$, then we have
\begin{equation}\label{SomUnif}
\sup_{y\in \R}\;\sum_{n\in \Z}\frac1{\< n-y\>^{2\gamma}}<\infty,
\end{equation}
and
\begin{equation}\label{EstKPV}
\sup_{(y,z)\in \R^{2}}\;\sum_{n\in \Z}\frac1{\< z+n(n-y)\>^{\gamma}}<\infty.
\end{equation}
\end{lemm}

\begin{proof}
$\bullet$ Let $y\in \R$. Up to a shift in $n$, we can assume that $y\in [0,1[$. Then $\< n-y\>\geq \frac12\< n\>$, hence the estimate \eqref{SomUnif}.\\[5pt]
$\bullet$ Denote by $r_{1}=r_{1}(y,z)$ and $r_{2}=r_{2}(y,z)$ the complex roots of the polynomial $z+X(X-y)$. Then 
\begin{equation*}
 z+n(n-y)=(n-r_{1})(n-r_{2}).
\end{equation*}
There are at most 10 indexes $n$ such that $|n-r_{1}|\leq 2$ or  $|n-r_{2}|\leq 2$. The remaining $n'$s satisfy
\begin{equation*}
\big\< (n-r_{1})(n-r_{2})\big\>\geq \frac12 \< n-r_{1}\>\< n-r_{2}\>.
\end{equation*}
Hence by the Cauchy-Schwarz inequality
\begin{equation*}
\sum_{n\in \Z}\frac1{\< z+n(n-y)\>^{\gamma}}\lesssim\Big( \sum_{n\in \Z}\frac1{\<n-r_{1} \>^{2\gamma}}\Big)^{\frac12}\Big(\sum_{n\in \Z}\frac1{\<n-r_{2} \>^{2\gamma}}\Big)^{\frac12},
\end{equation*}
which yields the result by  \eqref{SomUnif}.
\end{proof}

\begin{coro}\label{CorKPV}
If $\gamma_{1},\gamma_{2}>\frac12$, then 
\begin{equation}\label{EstCor1}
\sup_{(k,\tau)\in \Z\times \R}\;\sum_{n\in \Z}\frac1{\< -\tau+(n+k)^{2}+n^{2}\>^{\gamma_{1}}}<\infty,
\end{equation}
and
\begin{equation}\label{EstCor2}
\sup_{(m,k,\tau)\in \Z^{2}_{*}\times \R}\;\sum_{n\in \Z}\frac1{\< \tau-(n+k)^{2}+(n+m)^{2}+m^{2}\>^{2\gamma_{2}}}<\infty,
\end{equation}
where $\Z^{2}_{*}=\{(m,k)\in \Z^{2},\;\text{s.t.}\;m\neq k\}$.
\end{coro}

\begin{proof}
$\bullet$ We first prove the estimate \eqref{EstCor1}. For all $\tau, n,k$ we have 
\begin{equation*}
\<-\tau+(n+k)^{2}+n^{2}\>=\big\<-\tau+k^{2}+2n(n+k)\>\gtrsim\<\frac{-\tau+k^{2}}2+n(n+k)\big\>.
\end{equation*}
The estimate then follows from  \eqref{EstKPV} with $\gamma=\gamma_{1}>\frac12$, $y=-k$ and $z=(-\tau+k^{2})/2$.\\[5pt]
$\bullet$ We now turn to the proof of \eqref{EstCor2}. If $m\neq k$ are integers,  then $|m-k|\geq 1$ and thus
\begin{eqnarray*}
|\tau-(n+k)^{2}+(n+m)^{2}+m^{2}|&=&2|m-k|\big|\frac{\tau-k^{2}+2m^{2}}{2(m-k)}+n\big|\\
&\geq &|C+n|,
\end{eqnarray*}
with $C=(\tau-k^{2}+2m^{2})/(2(m-k))$. Therefore
\begin{equation*}
\<  \tau-(n+k)^{2}+(n+m)^{2}+m^{2}   \>\geq \<   n+C\>,
\end{equation*}
and the estimate follows from an application of  \eqref{SomUnif}.
\end{proof}

\begin{lemm}\label{LemmeInt}
If $\gamma>\frac12$, then 
\begin{equation*}
\sum_{n\in \Z}\frac1{\< n^{2}+y^{2}\>^{\gamma}}\lesssim \frac{1}{\<y\>^{2\gamma-1}}.
\end{equation*}
\end{lemm}

\begin{proof}
We can assume that $y>0$. We compare the sum with an integral, and with the change of variables $x=yt$ we obtain
\begin{eqnarray*}
\sum_{n\in \Z}\frac1{\< n^{2}+y^{2}\>^{\gamma}}&\lesssim &\sum_{n\in \N}\frac1{\< n^{2}+y^{2}\>^{\gamma}}\lesssim \int_{0}^{{+\infty}}\frac{\text{d}x}{\<x^{2}+y^{2}\>^{\gamma}}\\
&\lesssim&\frac{1}{\<y\>^{{2\gamma-1}}} \int_{0}^{{+\infty}}\frac{\text{d}t}{(t^{2}+1)^{\gamma}}\lesssim \frac{1}{\<y\>^{{2\gamma-1}}},
\end{eqnarray*}
which was the claim.
\end{proof}

\begin{proof}[Proof of Proposition \ref{bilinear}] Let $f\in\H^{s_{0},p}(\T)$ and write 
\begin{equation*}
f(x)=\sum_{n\in \Z}a_{n}\e^{inx}.
\end{equation*}
Denote by $u_{0}(t)=\e^{it\Delta}f$ the free Schr\"odinger evolution of $f$. Then
\begin{equation}\label{Q1}
u_{0}(t,x)=\e^{it\Delta}f(x)=\sum_{n\in \Z}a_{n}\e^{-in^{2}t}\e^{inx}.
\end{equation}
Let  $v\in X^{s,b}_{1}(\R\times \T)$, and let  $\psi_{0} \in \mathcal{C}^{\infty}_{0}(\R)$ be so that $\psi_{0}=1$ on $[-1,1]$ and $\text{supp}\;\psi_{0} \subset [-2,2]$. Moreover, we choose $\psi_{0}$ such that

\begin{equation}\label{Defpsi0}
\|v\|^{2}_{X^{s,b}_{1}}=\|\psi_{0}(t)\,v\|^{2}_{X^{s,b}}.
\end{equation}

\noindent Then we consider the following Fourier expansion
\begin{equation}\label{Q2}
\psi_{0}(t)\,v(t,x)=\sum_{n\in \Z}b_{n}(t)\e^{inx}.
\end{equation}
Thus by Definition \ref{defi} and \eqref{Defpsi0} we have
\begin{equation}\label{Normv}
\|v\|^{2}_{X^{s,b}_{1}}= \|\psi_{0}(t)\,v\|^{2}_{X^{s,b}}=\sum_{n \in \Z} \int_{\R}\<\tau+n^{2}\>^{2b}\<n\>^{2s}|\widehat{b}_{n}(\tau)|^{2}\text{d}\tau.
\end{equation}

\noindent Now, use the expressions \eqref{Q1} and \eqref{Q2} to  compute
\begin{eqnarray*}
\psi_{0}(t)\,u_{0}\,v(t,x)&=&\sum_{(j,k)\in \Z^{2}}a_{j}\,b_{k}(t)\e^{-itj^{2}}\e^{i(j+k)x}\\
&=&\sum_{n\in \Z}\Big(\sum_{k\in \Z}a_{-n-k}\,b_{k}(t)\e^{-it(n+k)^{2}}\Big)\e^{-inx},
\end{eqnarray*}
therefore 
\begin{equation}\label{decomp}
\psi_{0}(t)\,\ov{u_{0}}\,\ov{v}(t,x)=\sum_{n\in \Z}c_{n}(t)\e^{inx},
\end{equation}
with 
\begin{equation*}
c_{n}(t)=\sum_{k\in \Z}\ov{a_{-n-k}}\,\ov{b_{k}}(t)\e^{it(n+k)^{2}}.
\end{equation*}
Now from the properties \eqref{TF} of the time-Fourier transform, we  deduce

\begin{eqnarray*}
\widehat{c_{n}}(\tau)&=&\sum_{k\in \Z}\,\ov{a_{-n-k}}\;\widehat{  \ov { b_{k}(t)\e^{-it(n+k)^{2}} } }(\tau)
=\sum_{k\in \Z}\ov{a_{-n-k}}\;\ov {\widehat{  b_{k}(t)\e^{-it(n+k)^{2}} } }(-\tau)\\
&=&\sum_{k\in \Z} \ov{ a_{-n-k}}\;\ov{\widehat{b_{k}}}(-\tau+(n+k)^{2}).
\end{eqnarray*}
Now write
\begin{multline*}
\widehat{c_{n}}(\tau)=\\
\sum_{k\in \Z}\frac{ \ov{ a_{-n-k}}}{\<k\>^{s}\<-\tau+(n+k)^{2}+k^{2}\>^{b}}\;
\<k\>^{s}\<-\tau+(n+k)^{2}+k^{2}\>^{b}\ov{\widehat{b_{k}}}(-\tau+(n+k)^{2}),
\end{multline*}
and by the Cauchy-Schwarz inequality we obtain
\begin{equation}\label{CS}
|\widehat{c_{n}}(\tau)|^{2}\leq \Big(      \sum_{j\in\Z} A_{j,n}(\tau)\Big)
\Big( \sum_{k\in\Z}B_{k,n}(\tau)  \Big),
\end{equation}
where

\begin{equation}\label{defA}
A_{j,n}(\tau)=\frac{ |a_{-n-j  }  |^{2}}{\<j\>^{2s}\<-\tau+(n+j)^{2}+j^{2}\>^{2b}},
\end{equation}
and
\begin{equation}\label{defB}
B_{k,n}(\tau)=\<k\>^{2s}\<-\tau+(n+k)^{2}+k^{2}\>^{2b}|{\widehat{b_{k}}}|^{2}(-\tau+(n+k)^{2}).
\end{equation}

\noindent Now by Proposition \ref{EstInt}, for $\frac12<b<1$ and $s\in \R$

\begin{equation*}
\|\int_{0}^{t}\e^{i(t-t')\Delta}\ov{u_{0}}\,\ov{v}(t',\cdot)\text{d}t'\|_{X^{s,b}_{1}}  \lesssim \| \ov{u_{0}v}\|_{X^{s,b-1}_{1}}\leq 
\|\psi_{0}(t)\, \ov{u_{0}}\,\ov{v}\|_{X^{s,b-1}},
\end{equation*}
where the second inequality is a consequence of Definition \ref{defi}. \\[5pt]
Then by \eqref{decomp} and \eqref{CS} we obtain
\begin{eqnarray*}
\|\psi_{0}(t)\, \ov{u_{0}}\,\ov{v}\|^{2}_{X^{s,b-1}}&=&\sum_{n\in \Z}\int_{\R}\<\tau+n^{2}\>^{2(b-1)}\<n\>^{2s}|\widehat{c_{n}}(\tau)|^{2}\text{d}\tau\\
&\leq&\sum_{n\in \Z}\int_{\R}\frac{\<n\>^{2s}}{\<\tau+n^{2}\>^{2(1-b)}}\Big(      \sum_{j\in\Z} A_{j,n}(\tau)\Big)
\Big( \sum_{k\in\Z}B_{k,n}(\tau) \Big)\text{d}\tau\\
&=&\sum_{k\in \Z}\;\sum_{n\in \Z}\int_{\R}\Bigg(\sum_{j\in\Z}\frac{\<n\>^{2s}  A_{j,n}(\tau)}{\<\tau+n^{2}\>^{2(1-b)}}    \Bigg)
B_{k,n}(\tau) \text{d}\tau.
\end{eqnarray*}
Now, thanks to the change of variables $\tau'=-\tau+(n+k)^{2}$ and \eqref{defB} we deduce
\begin{multline*}
\| \psi_{0}(\frac{t}T)\, \ov{u_{0}}\,\ov{v}\|^{2}_{X^{s,b-1}}\leq\\
\begin{aligned}
&\leq\sum_{k\in \Z}\;\sum_{n\in \Z}\int_{\R}\Bigg(\sum_{j\in\Z}\frac{\<n\>^{2s}  A_{j,n}(-\tau'+(n+k)^{2})}{\<-\tau'+(n+k)^{2}+n^{2}\>^{2(1-b)}}    \Bigg)
B_{k,n}(-\tau'+(n+k)^{2}) \text{d}\tau'\\
&=\sum_{k\in \Z}{\int_{\R}}\;\Bigg(\sum_{(n,j)\in \Z^{2}}\frac{\<n\>^{2s}  A_{j,n}(-\tau'+(n+k)^{2})}{\<-\tau'+(n+k)^{2}+n^{2}\>^{2(1-b)}}    \Bigg)
\<k\>^{2s}\<\tau'+k^{2}\>^{2b}|{\widehat{b_{k}}}|^{2}(\tau') \text{d}\tau'\\
&\leq\sup_{(k,\tau)\in \Z\times \R}\Bigg[\;\sum_{(n,j)\in \Z^{2}}\frac{\<n\>^{2s}  A_{j,n}(-\tau+(n+k)^{2})}{\<-\tau+(n+k)^{2}+n^{2}\>^{2(1-b)}}    \Bigg]
 \sum_{k\in \Z} \int_{\R}\<k\>^{2s}\<\tau'+k^{2}\>^{2b}|{\widehat{b_{k}}}|^{2}(\tau') \text{d}\tau'\\
 &= \|v\|^{2}_{X^{s,b}_{1}}\sup_{(k,\tau)\in \Z\times \R}\Bigg[\;\sum_{(n,j)\in \Z^{2}}\frac{\<n\>^{2s}  A_{j,n}(-\tau+(n+k)^{2})}{\<-\tau+(n+k)^{2}+n^{2}\>^{2(1-b)}}    \Bigg],
 \end{aligned}
 \end{multline*}
by \eqref{Normv}. \\
It remains to estimate the term
\begin{equation*}
I(k,\tau):=\sup_{(k,\tau)\in \Z\times \R}\Bigg[\;\sum_{(n,j)\in \Z^{2}}\frac{\<n\>^{2s}  A_{j,n}(-\tau+(n+k)^{2})}{\<-\tau+(n+k)^{2}+n^{2}\>^{2(1-b)}}    \Bigg],
\end{equation*} 
uniformly in $(k,\tau)\in \Z\times \R$.\\
By the definition \eqref{defA} of $A_{j,n}$ and the change of indexes $m=-n-j$, we have
\begin{equation}\label{bound}
\begin{aligned}
&I(k,\tau)=\\
&=\sum_{(n,j)\in \Z^{2}}\frac{\<n\>^{2s}  |a_{-n-j  }  |^{2}}{\<j\>^{2s}\<-\tau+(n+k)^{2}+n^{2}\>^{2(1-b)}\<\tau-(n+k)^{2}+(n+j)^{2}+j^{2}\>^{2b} }\\ 
&=\sum_{(n,m)\in \Z^{2}}\frac{\<n\>^{2s}  |a_{m }  |^{2}}{\<n+m\>^{2s}\<-\tau+(n+k)^{2}+n^{2}\>^{2(1-b)}\<\tau-(n+k)^{2}+m^{2}+(n+m)^{2}\>^{2b} }\\
&:=\sum_{(n,m)\in \Z^{2}}I_{n,m}(k,\tau).
\end{aligned}
\end{equation}

\noindent Denote by 
\begin{equation*}
R_{1}=R_{1}(\tau,n,k)=-\tau+(n+k)^{2}+n^{2},
\end{equation*}
\begin{equation*}
 R_{2}=R_{2}(\tau,n,k,m)=\tau-(n+k)^{2}+m^{2}+(n+m)^{2}.
\end{equation*}
Denote by $\s=-s>0$ and $\s_{0}=-s_{0}\geq 0$. Write $b=\frac12+\eps$. Then introduce  
\begin{equation*}
\b_{1}=2(1-b)=1-2\eps<1\quad \text{and}\quad \b_{2}=2b=1+2\eps>1.
\end{equation*}
Therefore, $I_{n,m}$ can be rewritten
\begin{equation}\label{defI}
 I_{n,m}(k,\tau)=\frac{\<n+m\>^{2\s}}{\<n\>^{2\s}}\frac{|a_{m}|^{2}}{\<R_{1}\>^{\b_{1}}\<R_{2}\>^{\b_{2}}}.
\end{equation}

\noindent $\bullet$ Observe that $\b_{1}\leq \b_{2}$. Thus by \eqref{defI}, for all $m\neq k$ and $0\leq \theta\leq 1$
\begin{eqnarray} \label{maj}
\sum_{n\in \Z} I_{n,m}(k,\tau)&\leq&|a_{m}|^{2}\sum_{n\in \Z}   \frac{\<n+m\>^{2\s}}{\<n\>^{2\s}}\frac{1}{\<R_{1}\>^{\beta_{1}}}\frac{1}{\<R_{2}\>^{\beta_{1}}}\nonumber\\
&\leq&   |a_{m}|^{2} \sup_{n\in \Z}\Big[ \frac{\<n+m\>^{2\s}}{\<n\>^{2\s}}\frac{1}{\<R_{1}\>^{(1-\theta)\beta_{1}}}\frac{1}{\<R_{2}\>^{(1-\theta)\beta_{1}}}\Big]
\sum_{n\in \Z}\frac{1}{\<R_{1}\>^{\theta\beta_{1}}}\frac{1}{\<R_{2}\>^{\theta\beta_{1}}}.
\end{eqnarray}

\noindent For $p,q\geq 1$ such that $1/p+1/q=1$ we have the H\"older inequality
\begin{equation}\label{sums}
\sum_{n\in \Z}\frac{1}{\<R_{1}\>^{\theta\beta_{1}}}\frac{1}{\<R_{2}\>^{\theta\beta_{1}}}\leq \Big(\sum_{n\in \Z}\frac{1}{\<R_{1}\>^{\theta\beta_{1}p}}\Big)^{\frac1p}\Big(\sum_{n\in \Z}\frac{1}{\<R_{2}\>^{\theta\beta_{1}q}}\Big)^{\frac1q}.
\end{equation}
Now choose $p,q$ such that $\theta \b p=\frac12+\eps$ and $\theta \b q=1+2\eps$, i.e.
\begin{equation*}
p=\frac32,\quad q=3,\quad \text{and thus }\quad \theta=\frac{1+2\eps}{3(1-2\eps)}.
\end{equation*}
(Notice that $0\leq \theta\leq 1$ if $\eps>0$ is small enough). With these choices, by Corollary \ref{CorKPV}, all the sums in \eqref{sums} are uniformly bounded with respect to $(m,k,\tau)\in \Z^{2}_{*}\times \R$. Therefore, for $m\neq k$ we have
\begin{equation}\label{Inm}
\sum_{n\in \Z} I_{n,m}(k,\tau)\lesssim |a_{m}|^{2} \sup_{n\in \Z}\Big[ \frac{\<n+m\>^{2\s}}{\<n\>^{2\s}}\frac{1}{\<R_{1}\>^{(1-\theta)\beta_{1}}}\frac{1}{\<R_{2}\>^{(1-\theta)\beta_{1}}}\Big].
\end{equation}

\noindent Now we bound the $\sup_{n\in\Z}$ in \eqref{Inm}. Notice  that we have 
  the inequalities
\begin{equation*}
\frac{1}{\<R_{1}\>}\frac{1}{\<R_{2}\>}\leq \frac{1}{\<R_{1}+R_{2}\>}
=\frac{1}{\<n^{2}+m^{2}+(n+m)^{2}\>}\lesssim \frac{1}{\<m\>^{2}},
\end{equation*}
and $\<n+m\> \lesssim \<n\>\<m\>$. Hence
 \begin{equation}\label{eq1}
\sup_{n\in \Z}\Big[ \frac{\<n+m\>^{2\s}}{\<n\>^{2\s}}\frac{1}{\<R_{1}\>^{(1-\theta)\beta_{1}}}\frac{1}{\<R_{2}\>^{(1-\theta)\beta_{1}}}\Big]\lesssim  \frac{1}{\<m\>^{2(1-\theta)\beta_{1}-2\s}}.
\end{equation}
Then thanks to \eqref{eq1}, for $m\neq k$,  \eqref{Inm} becomes 
\begin{equation*}
\sum_{n\in \Z}I_{n,m}(k,\tau)\lesssim
 \frac{|a_{m}|^{2}}{\<m\>^{2(1-\theta)\b_{1} -2\s}}= \frac{|a_{m}|^{2}}{\<m\>^{\frac43(1-4\eps)-2\s}},
\end{equation*}
and by summing up, we obtain
\begin{equation}\label{pr}
\sum_{(n,m)\in \Z^{2},m\neq k}I_{n,m}(k,\tau)\lesssim
\sum_{m\in \Z} \frac{ |a_{m }  |^{2}}{\<m\>^{\frac43(1-4\eps)-2\s}}=\sum_{m\in \Z} \frac{ |a_{m }  |^{2}}{\<m\>^{2\s_{0}}}\frac{ 1}{\<m\>^{\eta}},
\end{equation}
with
\begin{equation}\label{eta}
\eta=\frac43(1-4\eps)-2\s_{0}-2\s.
\end{equation}
Now apply  H\"older to \eqref{pr} : For all $p\geq 2$ and $1/q=1-2/p$ so that $q\eta>1$, we can write
\begin{eqnarray*}
\sum_{(n,m)\in \Z^{2},m\neq k}I_{n,m}(k,\tau)&\lesssim&\Big(\sum_{m\in \Z} \frac{ |a_{m }  |^{p}}{\<m\>^{\s_{0}p}}\Big)^{\frac2p}\Big(\sum_{j\in \Z} \frac{1}{\<j\>^{q\eta}}\Big)^{\frac1q}.
\end{eqnarray*}
By \eqref{eta}, the condition $q\eta>1$ is equivalent to 

\begin{equation*}
\frac43(1-4\eps)-2\s_{0}-2\s=\eta>\frac1q=1-\frac2p,
\end{equation*}
or 
\begin{equation}\label{condsigma}
\s< \frac16-\s_{0}+\frac1p  -\frac83\eps.
\end{equation}
Assume that \eqref{condprop} is satisfied. Then for $0<\eps\leq \eps_{1}$ (for $\eps_{1}$ small enough), the condition \eqref{condsigma} is also satisfied and we have
\begin{equation*}
\sum_{(n,m)\in \Z^{2},m\neq k}I_{n,m}(k,\tau)\lesssim
\|f\|^{2}_{H^{s_{0},p}}.
\end{equation*}

\noindent $\bullet$ We now consider the case $m=k$. \\[5pt]
By \eqref{bound}, we have to bound, uniformly in $(k,\tau)\in \Z\times \R$, the term
\begin{equation*}
\sum_{n\in \Z}I_{n,k}(k,\tau)=|a_{k }  |^{2}\sum_{n\in \Z}\frac{\<n+k\>^{2\s}}{\<n\>^{2\s}}\frac{ 1 }{\<-\tau+(n+k)^{2}+n^{2}\>^{\b_{1}}\<\tau+k^{2}\>^{\b_{2}} }.
\end{equation*}
By the  inequality $\<a+b\>\leq \<a\>\<b\>$ and Lemma \ref{LemmeInt} we obtain (recall that $\b_{1}=1-2\eps$)
\begin{eqnarray*}
\sum_{n\in \Z}I_{n,k}(k,\tau)&\leq&|a_{k }  |^{2}\sum_{n\in \Z}\frac{\<n+k\>^{2\s}}{\<n\>^{2\s}}\frac{ 1 }{\<k^{2}+(n+k)^{2}+n^{2}\>^{\b_{1}}}\\
& \leq&|a_{k }  |^{2}\sum_{n\in \Z}\frac{ \<n+k\>^{2\s}}{\<n\>^{2\s}\<k^{2}+n^{2}\>^{1-2\eps}}\\
&\lesssim&|a_{k }  |^{2}\sum_{n\in \Z}\frac{1}{\<n\>^{2\s}\<k^{2}+n^{2}\>^{1-\s-2\eps}}\\
&\lesssim&|a_{k }  |^{2}\sum_{n\in \N}\frac{1}{\<n\>^{2\s}\<k^{2}+n^{2}\>^{1-\s-2\eps}}.
\end{eqnarray*}
Now we compare this sums with an integral : Thanks to the change of variables $x=|k|\,y$ we obtain, as  $\s<\frac12$ 
\begin{eqnarray*}
\sum_{n\in \Z}I_{n,k}(k,\tau)&\lesssim & |a_{k }  |^{2}\int_{0}^{+\infty}\frac{\text{d}x}{\<x\>^{2\s}\<k^{2}+x^{2}\>^{1-\s-2\eps}}\\
&\lesssim &\frac{|a_{k }  |^{2}}{\<k\>^{1-4\eps}}\int_{0}^{+\infty}\frac{\text{d}y}{y^{2\s}\<1+y^{2}\>^{1-\s-2\eps}}\\
&\lesssim& \frac{|a_{k }  |^{2}}{\<k\>^{1-4\eps}}\lesssim \|f\|^{2}_{H^{s_{0},p}}\;,
\end{eqnarray*}
whenever $1-4\eps\geq 2\s_{0}=-2s_{0}$, i.e. for $0<\eps\leq \eps_{2}$.\\
Finally, set $b_{2}=\frac12+\eps$, with $\eps=\min{(\eps_{1},\eps_{2})}$. This concludes the proof.
\end{proof}

\section{Proof of the main theorem}\label{contrac}

We now have all the ingredients to prove Theorem \ref{theo2} (observe that  Theorem \ref{theo1} is a particular case of the latter).

\begin{proof}[Proof of Theorem \ref{theo2}] To take profit of the gain of regularity of the first Picard iterate ( Proposition \ref{propPicard}) we write $u=\e^{it\Delta}f+v$ and where $v$ lives in a smaller space than $u$.  This idea was used by N. Burq and N. Tzvetkov \cite{BT2, BT3} in the context of supercritical wave equations. \\
We plug this expression in the integral equation
\begin{equation*}
u=\e^{it\Delta}f-i\kappa\int_{0}^{t}\e^{i(t-t')\Delta}(\overline{u}^{2})(t',x)\text{d}t',
\end{equation*}
then we will show  that the map $K$ defined by
\begin{eqnarray*}
K(v)&=&-i\kappa\int_{0}^{t}\e^{i(t-t')\Delta}(\overline{u_{0}}^{2})(t',x)\text{d}t'-2i\kappa\int_{0}^{t}\e^{i(t-t')\Delta}\ov{u_{0}}\,\ov{v}(t',\cdot)\text{d}t'\\
&&-i\kappa\int_{0}^{t}\e^{i(t-t')\Delta}(\overline{v}^{2})(t',x)\text{d}t',
\end{eqnarray*}
is a contraction.\\
Let $p\geq 2$ and $s_{0}>-\frac12$ satisfy the condition \eqref{cond}, i.e.
\begin{equation*}
\frac3p+s_{0}>\frac56,
\end{equation*}
then there exists $s>-\frac12$ so that
\begin{equation*}
-\frac16-s_{0}-\frac1p<s<-1+\frac2p,
\end{equation*}
and we can use the estimates \eqref{FondaKPV}, \eqref{EstPic} and \eqref{Estproduit} to obtain : There exist $b>\frac12$ and $C\geq1$ such that
\begin{equation}\label{estK}
\|K(v)\|_{X^{s,b}_{1}}  \leq C\big(  \|f\|^{2}_{\H^{s_{0},p}}   + \|f\|_{\H^{s_{0},p}}\|v\|_{X^{s,b}_{1}}+\|v\|^{2}_{X^{s,b}_{1}}      \big),
\end{equation}
and
\begin{equation}\label{estK1}
\|K(v_{1})-K(v_{2})\|_{X^{s,b}_{1}}  \leq C\big( \|f\|_{\H^{s_{0},p}}+\|v_{1}+v_{2}\|_{X^{s,b}_{1}}     \big)\|v_{1}-v_{2}\|_{X^{s,b}_{1}}.
\end{equation}
~\\[4pt]
\noindent $\bullet$ The case of small initial data. We assume that $\|f\|_{\H^{s_{0},p}}=\mu\ll 1$. Then we show that $K$ is a contraction on the ball of radius $C\mu$ in $X^{s,b}$, for $\mu$ small enough. For $\|v_{1}\|_{X^{s,b}},\|v_{2}\|_{X^{s,b}+1}\leq C\mu $, we deduce from  \eqref{estK} and \eqref{estK1} that  
\begin{equation*}
\|K(v)\|_{X^{s,b}_{1}}  \leq C\big(\mu^{2}   + \mu\|v\|_{X^{s,b}_{1}}+\|v\|^{2}_{X^{s,b}_{1}}      \big)\leq 3C^{2}\mu^{2},
\end{equation*}
and
\begin{equation*}
\|K(v_{1})-K(v_{2})\|_{X^{s,b}_{1}}  \leq C\big( \mu+\|v_{1}+v_{2}\|_{X^{s,b}_{1}}     \big)\|v_{1}-v_{2}\|_{X^{s,b}}\leq 3C^{2}\mu\|v_{1}-v_{2}\|_{X^{s,b}_{1}},
\end{equation*}
and the result follows if we choose $\mu$ so that $3C^{2}\mu<1$.\\[2pt]
The argument to show the uniqueness of the solution in the whole space is similar to the argument given in \cite{KPV}, we do not give more details here.\\[5pt]
$\bullet$ The general case. Let $u$ be a solution of \eqref{NLS}, then for all $\lambda>0$, $u_{\lambda}$ defined by $u_{\lambda}(t,x)=\lambda^{2}u(\lambda^{2}t,\lambda x)$ in also a solution of the equation, but on a torus of period $2\pi/{\lambda}$. It is easy to check that the estimates \eqref{FondaKPV}, \eqref{EstPic} and \eqref{Estproduit} still hold uniformly w.r.t $\lambda>0$, if we replace $\R/(2\pi\Z)$ with   $\R/(\frac{2\pi}{\lambda}\Z)$ (see Molinet \cite{Molinet} for more details). Now as 
\begin{equation*}
\|f_{\lambda}\|_{\H^{s_{0},p}}=\|u_{\lambda}(0,\cdot)\|_{\H^{s_{0},p}}\sim \lambda^{1+s_{0}+\frac1p},
\end{equation*}
which tends to 0, we can apply the result of the previous case, and find  a unique solution $u\in X^{s,b}([-\lambda^{2},\lambda^{2}]\times \T)$, for $\lambda$ small enough.\\[5pt]
$\bullet$ The argument showing the regularity of the flow map is exactly the same as in \cite{KPV}, hence we omit the proof here.
\end{proof}

\begin{rema}
We may compute  the following Picard iterates of $u$. Therefore we could look for a solution to \eqref{NLS} of the form $u=u_{0}+u_{1}+\dots+u_{n}+v$, where the $u_{j}$'s are known explicitly and where the unknown $v$ in more regular than $u_{n}$.  A fixed point argument on $v$ would improve a bit the range \eqref{cond}. However we do not pursue this  strategy as we do not think this will give an optimal result.
\end{rema}

\begin{rema}
The conclusion of Theorem \ref{theo2} may be improved using estimates in $X^{s,b}_{{p,q}}$ space, i.e. $X^{s,b}$ spaces based on $L^{p}$ in the space frequency variable and $L^{q}$ in the variable $\tau$. See \cite{GrunHerr} for such a strategy for the DNLS equation. 
\end{rema}



\begin{thebibliography}{9}


 
\bibitem{BejenaruTao}
I.~Bejenaru, and T. Tao.
\newblock Sharp well-posedness and ill-posedness results for a quadratic non-linear Schr\"odinger equation.
\newblock {\em arXiv:math/0508210.}
\bibitem{Bourgain1:res}
J.~Bourgain.
\newblock Fourier transform restriction phenomena for certain lattice subsets
  and applications to nonlinear evolution equations. {I}. {S}chr\"odinger
  equations.
\newblock {\em Geom. Funct. Anal.}, 3(2):107--156, 1993.




\bibitem{BGT2}
N.~Burq, P.~G{\'e}rard and  N.~Tzvetkov.
\newblock  Eigenfunction estimates and the Nonlinear Schr\"odinger equations on surfaces.
\newblock {\em Invent. Math.},  159, no. 1, 187--223 126, 2005.



 

 \bibitem{BT2}
 N.~Burq and  N.~Tzvetkov.
 \newblock Random data Cauchy theory for supercritical wave equations  \nolinebreak[4 ] I: local
existence theory.
 \newblock{\em Invent. Math.} 173, No. 3, 449-475 (2008)

\bibitem{BT3}
 N.~Burq and  N.~Tzvetkov.
 \newblock Random data Cauchy theory for supercritical wave equations \nolinebreak[4 ] II:  A global existence result.
 \newblock{\em Invent. Math.} 173, No. 3, 477-496 (2008) 
 
 \bibitem{CaVeVi}
T.~Cazenave, L.~Vega and  M.C~Vilela.
\newblock A note on the nonlinear Schr\"odinger equation in weak $L^{p}$ spaces.
\newblock {\em Commun. Contemp. Math.} 3(1) : 153--162, 2001. 


\bibitem{CW}
T.~Cazenave and  F.~B.~Weissler.
\newblock The Cauchy problem for the critical nonlinear Schr\"odinger equation in $H\sp s$.
\newblock {\em Nonlinear Anal.} 14, no. 10, 807--836, 1990. 


\bibitem{Christ}
M.~Christ.
\newblock Power series of a nonlinear Schr\"odinger equation. 
\newblock {\em  Mathematical aspects of nonlinear dispersive equations,}  131--155, Ann. of Math. Stud., 163, Princeton Univ. Press, Princeton, NJ, 2007.

\bibitem{Ginibre}
 J.~Ginibre.
\newblock  Le probl\`eme de Cauchy pour des EDP semi-lin\'eaires p\'eriodiques en variables d'espace (d'apr\`es Bourgain).  
S\'eminaire Bourbaki, Vol. 1994/95.
\newblock {\em Ast\'erisque} No. 237 (1996), Exp. No. 796, 4, 163--187. 

\bibitem{GV}
 J.~Ginibre and   G.~Velo.
\newblock  The global Cauchy problem for the nonlinear Schr\"o\-dinger
equation. 
\newblock {\em Ann. I.H.P. Anal. non lin.}, 2:309--327, 1985.



\bibitem{Grun1}
 A.~Gr\"unrock.
\newblock An improved local well-posedness result for the modified KdV equation.
\newblock {\em  Int. Math. Res. Not. } 2004,  no. 61, 3287--3308.

\bibitem{Grun2}
 A.~Gr\"unrock.
\newblock Bi- and trilinear Schr\"odinger estimates in one space dimension with applications to cubic NLS and DNLS.
\newblock {\em  Int. Math. Res. Not. }    2005,  no. 41, 2525--2558.



\bibitem{GrunHerr}
 A.~Gr\"unrock and S.~Herr.
\newblock Low regularity local well-posedness of the derivative nonlinear Schr\"odinger equation with periodic initial data.
\newblock {\em  SIAM J. Math. Anal. } 39  (2008),  no. 6, 1890--1920.

\bibitem{Hormander}
L.~H\"ormander.
\newblock The analysis of linear partial differential operators. II. Differential operators with constant coefficients. 
\newblock {\em Grundlehren der Mathematischen Wissenschaften}, 257. Springer-Verlag, Berlin, 1983.

\bibitem{KPV}
 C.~E.~Kenig, G.~Ponce, and L.~Vega.
\newblock Quadratic forms for the $1$-D semilinear Schr\"odinger equation.
\newblock {\em Trans. Amer. Math. Soc. }348 (1996), no. 8, 3323--3353. 



\bibitem{Molinet}
L.~Molinet.
\newblock Global well-posedness in the energy space for the Benjamin-Ono equation on the circle.
\newblock {\em  Math. Ann.}, (2007), 337: 353--383.





\bibitem{Tsutsumi}
Y.~Tsutsumi.
\newblock  $L^{2}$-solutions for nonlinear Schr\"odinger equations ond nonlinear groups, 
\newblock {\em  Funk. Ekva.} 30 (1987), 115--125.

\bibitem{Tzvetkov2}
N.~Tzvetkov.
\newblock  Invariant measures for the defocusing NLS.\\
\newblock {\em  arXiv:math/0701287.  }



\end{thebibliography}
\end{document}